\documentclass[12pt,twoside]{amsart}
\usepackage{amsmath}
\usepackage{amsthm}
\usepackage{amsfonts}
\usepackage{amssymb}
\usepackage{latexsym}
\usepackage{mathrsfs}
\usepackage{amsmath}
\usepackage{amsthm}
\usepackage{amsfonts}
\usepackage{amssymb}
\usepackage{latexsym}
\usepackage{geometry}
\usepackage{dsfont}
\usepackage[dvips]{graphicx}
\usepackage{color}
\usepackage[all]{xy}

\date{}
\allowdisplaybreaks[4] \footskip=15pt
\renewcommand{\uppercasenonmath}[1]{}

\usepackage{graphicx,amssymb}
\usepackage[all]{xy}
\usepackage{amsmath}

\allowdisplaybreaks
\usepackage{amsthm}
\usepackage{color}

\theoremstyle{plain}
\newtheorem{theorem}{Theorem}
\newtheorem{proposition}[theorem]{Proposition}
\newtheorem{lemma}[theorem]{Lemma}

\newtheorem{example}[theorem]{Example}
\newtheorem*{open question}{Open Question}

\theoremstyle{definition}

\theoremstyle{remark}
\newtheorem{remark}[theorem]{Remark}

\begin{document}
\begin{center}
{\large  \bf A note on $S$-coherent rings}

\vspace{0.5cm}    Xiaolei Zhang$^{a}$

{\footnotesize a.\  School of Mathematics and Statistics, Shandong University of Technology, Zibo 255000, China\\

E-mail: zxlrghj@163.com\\}
\end{center}

\bigskip
\centerline { \bf  Abstract}
\bigskip
\leftskip10truemm \rightskip10truemm \noindent

In this note, we show that  a ring $R$ is $S$-coherent if and only if every finitely
presented $R$-module is $S$-coherent,  providing a positive answer to a question proposed in [D. Bennis, M. El Hajoui, {\it On $S$-coherence}, J. Korean Math. Soc. \textbf{55} (2018), no.6, 1499-1512]. Besides, we show that  $c$-$S$-coherent rings are $S$-coherent, and give an example to show the converse is not true in general.
\vbox to 0.3cm{}

\leftskip0truemm \rightskip0truemm
\bigskip

Throughout this paper, $R$ is a commutative ring with identity and all modules are unitary. $S$ will always denote a multiplicative closed set of $R$. Let $M_1, M_2$ be two subsets of an $R$-module $M$. Set $(M_1:_RM_2)=\{r\in R\mid rM_2\subseteq M_1\}.$  For a subset $U$ of  an $R$-module $M$, denote by $\langle U\rangle$ the submodule of $M$ generated by $U$.

 In 2002, Anderson and Dumitrescu \cite{ad02} introduced the notions of $S$-finite modules and $S$-Noetherian rings. An $R$-module $M$ is called \emph{$S$-finite} provided that there exists a finitely generated submodule $N$ of $M$ such that $$sM\subseteq N\subseteq M$$ for some $s\in S$. And a ring $R$ is called an \emph{$S$-Noetherian ring} if every ideal of $R$ is $S$-finite. Some classical results, such as Cohen's Theorem, Eakin-Nagata Theorem and Hilbert Basis Theorem, were investigated for $S$-Noetherian rings \cite{ad02}.

Coherent rings are important generalizations of Noetherian rings. In 2018, Bennis and Hajoui \cite{bh18} introduced the notions of $S$-finitely presented modules and $S$-coherent modules which can be seen as $S$-versions of finitely presented modules and coherent modules. An $R$-module $M$ is said to be \emph{$S$-finitely presented} provided  there exists
an exact sequence of $R$-modules $$0 \rightarrow K\rightarrow F\rightarrow M\rightarrow 0,$$ where $K$ is
$S$-finite and $F$ is  finitely generated free. An $R$-module $M$ is said to be \emph{$S$-coherent} if it is finitely generated and every finitely generated submodule of $M$ is $S$-finitely presented. A ring $R$ is called an \emph{$S$-coherent ring} if $R$ itself is an $S$-coherent $R$-module, that is, every finitely generated ideal of $R$ is $S$-finitely presented.

The authors \cite{bh18} obtained an $S$-version of Chase's result \cite[Theorem 2.2]{c} for $S$-coherent rings as bellow.
 \begin{theorem}\label{B S-chase}
\textbf{ \cite[Theorem 3.8]{bh18}} Let $R$ be an ring and $S$  a multiplicative closed set of $R$.
Then the following assertions are equivalent:
\begin{enumerate}
    \item $R$ is an $S$-coherent ring;
    \item $(I:_Ra)$ is an $S$-finite ideal of $R$, for every finitely generated ideal $I$ of $R$ and $a\in R$;
     \item  $(0:_Ra)$ is an $S$-finite ideal of $R$ for every $a\in R$ and the intersection of  two finitely generated ideals of $R$ is an $S$-finite ideal of $R$.
\end{enumerate}
\end{theorem}

Subsequently, they asked of an
$S$-version of flatness that characterizes $S$-coherent rings similarly to the classical case. Recently, the author, Qi and Zhao \cite{qwz23} gave an $S$-version of Chase Theorem in terms of $S$-flat modules, i.e., modules $M$ satisfying $M_S$ is a flat $R_S$-module.

 \begin{theorem}\cite[Theorem 4.4]{qwz23} Let $R$ be an ring and $S$  a multiplicative closed set of $R$. Then the following assertions are equivalent:
\begin{enumerate}
    \item $R$ is an $S$-coherent ring;
    \item any product of flat $R$-modules is $S$-flat;
       \item any product of  projective $R$-modules is $S$-flat;
        \item any product of copies of  $R$ is $S$-flat.
\end{enumerate}
\end{theorem}
It is well-known that a ring $R$ is coherent if and only if every finitely
presented $R$-module is coherent (see \cite[Theorem 2.3.2]{g}).  
Immediately after the question of $S$-version of Chase Theorem, Bennis and Hajoui \cite{bh18} also left the $S$-version of the  result \cite[Theorem 2.3.2]{g}  as an open question.
\begin{open question}
Does the condition
``a ring $R$ is $S$-coherent'' imply $($and then equivalent to$)$ the condition ``every finitely
presented $R$-module is $S$-coherent''?
\end{open question}

In this note, we will give an positive answer to this question.

\begin{theorem}$(=${\bf Theorem \ref{main}}$)$ A ring $R$ is an $S$-coherent ring if and only if every finitely presented $R$-module is $S$-coherent.
\end{theorem}

We first recall some $S$-finitely presented properties on exact sequences.

\begin{lemma}\label{exac} \cite[Theorem 2.4]{bh18} Let $0 \rightarrow M' \rightarrow M \rightarrow M''\rightarrow 0$ be an exact sequence of	$R$-modules. The following assertions hold:
	\begin{enumerate}
		\item If $M'$ and $M''$ are $S$-finitely presented, then $M$ is $S$-finitely presented;
		In particular, every finite direct sum of $S$-finitely presented modules is
		$S$-finitely presented;
		\item If $M'$ is $S$-finite and $M$ is $S$-finitely presented, then  $M''$ is $S$-finitely
		presented;
		\item If $M''$ is $S$-finitely presented and $M$ is $S$-finite, then $M'$ is $S$-finite.
	\end{enumerate}
\end{lemma}

\begin{lemma}\label{cap} Let $R$ be an $S$-coherent ring. Then every finite intersection of   $S$-finite ideals of $R$ is $S$-finite.
\end{lemma}
\begin{proof}  Let $\{I_i\mid i=1,\dots,n\}$ be a finite family of $S$-finite ideals of $R$. Then, for each $i$, there exists a finitely generated sub-ideal $J_i$ of $I_i$ such that $s_iI_i\subseteq J_i\subseteq I_i$ for some $s_i\in S$.
Suppose $n=2$.	 Consider the following  exact sequences of commutative diagram:
 $$\xymatrix@R=25pt@C=30pt{ &0 \ar[d]^{}&0 \ar[d]^{} &0 \ar[d]^{}&\\
	0 \ar[r]^{} & J_1\cap J_2 \ar[d]\ar[r]^{} &J_1\oplus J_2 \ar[d]\ar[r]^{} &J_1+J_2\ar[d]\ar[r]^{} &  0\\
	0 \ar[r]^{} & I_1\cap I_2\ar[r]^{} & I_1\oplus I_2 \ar[r]^{} &I_1+I_2\ar[r]^{} &  0\\
&\\}$$
Since $R$ is an $S$-coherent ring and $J_1, J_2$ are finitely generated, $J_1\cap J_2$ is $S$-finite by Theorem \ref{B S-chase}. Set $s=s_1s_2$. Then it is easy to verify $s(I_1\cap I_2)\subseteq J_1\cap J_2$. Consequently, $I_1\cap I_2$ is also $S$-finite. Suppose it holds for $n\leq k$. Suppose $n=k+1$. Then there is an exact sequence $$0\rightarrow (\bigcap\limits_{i=1}^kI_i)\cap I_{k+1}\rightarrow (\bigcap\limits_{i=1}^kI_i)\oplus I_{k+1}\rightarrow (\bigcap\limits_{i=1}^kI_i)+I_{k+1}\rightarrow 0.$$ Similar to the $n=2$ case, one can show  $\bigcap\limits_{i=1}^{k+1}I_i=(\bigcap\limits_{i=1}^kI_i)\cap I_{k+1}$ is also $S$-finite.
\end{proof}

\begin{proposition}\label{freescoh}  Let $R$ be an $S$-coherent ring. Then every finitely generated free $R$-module is an $S$-coherent $R$-module.
\end{proposition}
\begin{proof} Let $M$ be a finitely generated submodule of a finitely generated free $R$-module $F=R^{n}$. Assume $M$ is generated by $m$ elements. We will show $M$ is $S$-finitely presented by induction on $m$. If $m$=1, then $M=Rx$ for some $x\in M$. So there is an exact sequence $$0\rightarrow (0:_Rx)\rightarrow R\rightarrow M\rightarrow0.$$ We will show $(0:_Rx)$ is $S$-finite by induction on $n$. If $n=1$, then the result  follows by Theorem \ref{B S-chase}. Suppose $x=(x_1,\dots,x_n)\in  M\subseteq R^n$. Then $$(0:_Rx)=\bigcap\limits_{i=1}^n(0:_Rx_i).$$ It follows by Theorem \ref{B S-chase} that $(0:_Rx_i)$ is $S$-finite for each $i=1,\dots,n$. It follows by Lemma \ref{cap} that $(0:_Rx)$ is also $S$-finite. Consequently, $M$ is $S$-finitely presented by Lemma \ref{exac}(2).
	
 Assume that $M$ is $S$-finitely presented for  all $m$-generated submodule $M$ of $R^n$ with all $m\leq k$. Now suppose $M=\langle m_1,\dots,m_k,m_{k+1}\rangle$. Set  $M_k=\langle m_1,\dots,m_k\rangle$.	Consider the following commutative diagram of exact sequences:
 
	 $$\xymatrix@R=25pt@C=30pt{ &0 \ar[d]^{}&0 \ar[d]^{} &0 \ar[d]^{}&\\
		0 \ar[r]^{} & L\cap R^k\ar[d]\ar[r]^{} &R^{k} \ar[d]\ar[r]^{} &M_k \ar[d]\ar[r]^{} &  0\\
		0 \ar[r]^{} & L\ar[d]\ar[r]^{} & R^{k+1} \ar[d]\ar[r]^{} &M\ar[d]\ar[r]^{} &  0\\
		0 \ar[r]^{} & (M_k:_Rm_{k+1})\ar[d]^{} \ar[r]^{} & R\ar[r]^{}\ar[d]^{}  & M/M_k\ar[r]^{} \ar[d]^{}&  0\\
	 &0 &0 &0 &\\}$$
Since $M_k$ is a $k$-generated submodule of $M$,	we have $M_k$ is $S$-finitely presented by induction. We will show that $(M_k:_Rm_{k+1})$ is $S$-finite by induction on $n$. Write $m_{k+1}=(x_1,\dots,x_n)\in M_k\subseteq R^n$. If $n=1$, then $M_k$ is a finitely generated ideal of $R$ and $m_{k+1}$ is an element in $R$. So $(M_k:_Rm_{k+1})$ is $S$-finite by Theorem \ref{B S-chase}. Now suppose $n>1$. Then $$(M_k:_Rm_{k+1})=\bigcap\limits_{i=1}^n\ (M_ke_i:_Rx_{i}),$$ where $M_ke_i$ is the $i$-th component of $M_k$ in $R^n$. Then $M_ke_i$  is a finitely generated ideal of $R$. It follows  by Theorem \ref{B S-chase} and Lemma \ref{cap} that $(M_k:_Rm_{k+1})$ is $S$-finite. Consequently, $M/M_k$ is $S$-finitely presented by Lemma \ref{exac}(2). It follows  by Lemma \ref{exac} (1) that $M$ is an $S$-finitely presented $R$-module. Consequently, $F$ is an $S$-coherent $R$-module.
\end{proof}

Now we are ready to prove the main result of this note.
\begin{theorem}\label{main} A ring $R$ is an $S$-coherent ring if and only if every finitely
presented $R$-module is $S$-coherent.
\end{theorem}
\begin{proof} The sufficiency is trivial. For necessity, let $M$ be a finitely presented $R$-module, and $L$ be a finitely generated submodule of $M$. Consider the following pull-back of $R$-modules:
	
 $$\xymatrix@R=25pt@C=30pt{ &&0 \ar[d]^{} &0 \ar[d]^{}&\\
		0 \ar[r]^{} & K\ar@{=}[d]\ar[r]^{} &X \ar[d]\ar[r]^{} &L \ar[d]\ar[r]^{} &  0\\
		0 \ar[r]^{} & K\ar[r]^{} & F \ar[d]\ar[r]^{} &M\ar[d]\ar[r]^{} &  0\\
	 &  & Y\ar@{=}[r]^{}\ar[d]^{}  & Y \ar[d]^{}& \\
		& &0 &0 &\\}$$
	where $F$ is a finitely generated free $R$-module and $K$ is a finitely generated submodule of $F$. Since $R$ is $S$-coherent, $F$ is an $S$-coherent $R$-module by Proposition \ref{freescoh}. Since $K$ and $L$ are finitely generated, so is $X$. Hence $X$ is $S$-finitely presented since $F$ is $S$-coherent. It follows by Lemma \ref{exac}(2) that $L$ is also $S$-finitely presented. Consequently, $M$ is an $S$-coherent $R$-module.
\end{proof}

\begin{remark} Recall from \cite{zuscoh-24} that an  $R$-module $M$ is called
	\emph{$u$-$S$-finitely presented} (abbreviates \emph{uniformly  $S$-finitely presented}) (with respect to $s$) provided that there is an exact sequence $$0\rightarrow T_1\rightarrow F\xrightarrow{f} M\rightarrow T_2\rightarrow 0$$ with $F$ finitely presented and $sT_1=sT_2=0$.  An  $R$-module $M$ is called a
	\emph{$u$-$S$-coherent module} (with respect to $s$) provided that there is $s\in S$ such that it is $S$-finite with respect to $s$ and  any finitely generated submodule of $M$ is $u$-$S$-finitely presented with respect to $s$. A ring   $R$ is called a
	\emph{$u$-$S$-coherent ring}  (with respect to $s$) provided that $R$ itself is a uniformly  $S$-coherent $R$-module with respect to $s$.  
\end{remark} 
\begin{theorem}
A ring $R$ is an $u$-$S$-coherent ring if and only if there exists $s\in S$ such that  every finitely	presented $R$-module is $u$-$S$-coherent with respect to $s$.
\end{theorem}
\begin{proof}
	It is similar to the proof of Theorem \ref{main}, and so we omit it.
\end{proof}

\begin{remark}
	Following \cite[Definition 4.1]{bh18} that an  $R$-module $M$ is called $c$-$S$-finitely presented
	provided that there exists a finitely presented submodule $N$ of $M$ such that $sM\subseteq N\subseteq M$ for some $s\in S$. An $R$-module $M$ is called $c$-$S$-coherent if it is finitely generated and every finitely generated submodule of $M$ is $c$-$S$-finitely presented.  Recall from \cite{bh18} that a ring $R$ is  \emph{$c$-$S$-coherent} provided that any finitely generated ideal is  $c$-$S$-finitely presented, equivalently,  $R$ is a $c$-$S$-coherent $R$-module.  Coherent rings are certainly $c$-$S$-coherent rings. However, the converse is not true in general. Indeed, let $D$ be a non-coherent domain and $S=D-\{0\}$. Then for any nonzero finitely generated ideal $I$ of $D$, we have $D\cong sD\subseteq I$ with $0\not= s\in I$. Hence $I$ is  $c$-$S$-finitely presented. Consequently, $D$ is a $c$-$S$-coherent ring.
\end{remark}	
	The rest of this note is devote to give some connections between  $S$-coherent rings and $c$-$S$-coherent rings.
\begin{lemma}\label{exccs} Let $0\rightarrow K\rightarrow F\rightarrow M\rightarrow 0$ be an exact sequence with $M$  $c$-$S$-finitely presented $R$-module and $F$ finitely generated free. Then $K$ is $S$-finite.
\end{lemma}	
\begin{proof} Note that $M$ is finitely generated and  $c$-$S$-finitely presented. Then there exists a finitely presented submodule  $N$ of $M$ such that $s(M/N)=0$. So  $M$ is $u$-$S$-finitely presented. It follows by \cite[Theorem 2.2(4)]{zuscoh-24} that $K$ is $S$-finite.
\end{proof}

\begin{theorem}
Suppose $R$ is a  $c$-$S$-coherent ring. Then $R$ is an  $S$-coherent ring.

\end{theorem}
\begin{proof} Assume $R$ is a  $c$-$S$-coherent ring.  Let $I=\langle x_1,\dots,x_n\rangle$ be a finitely generated ideal of $R$ and $a$ an element in $R$. Set $J = I + Ra$. Then $J$ is finitely generated,	and so it is also $c$-$S$-finitely presented.
Consider the following commutative diagram of exact sequences:

$$\xymatrix@R=25pt@C=30pt{ &0 \ar[d]^{}&0 \ar[d]^{} &0 \ar[d]^{}&\\
	0 \ar[r]^{} & K\cap R^n\ar[d]\ar[r]^{} &R^{n} \ar[d]\ar[r]^{} &I \ar[d]\ar[r]^{} &  0\\
	0 \ar[r]^{} & K\ar[d]\ar[r]^{} & R^{n+1} \ar[d]\ar[r]^{} &J\ar[d]\ar[r]^{} &  0\\
	0 \ar[r]^{} & (I:_Ra)\ar[d]^{} \ar[r]^{} & R\ar[r]^{}\ar[d]^{}  & J/I\ar[r]^{} \ar[d]^{}&  0\\
	&0 &0 &0 &\\}$$	
By Lemma \ref{exccs} we have $K$ is $S$-finite. So $(I:_Ra)$ is also $S$-finite. It follows by Theorem \ref{B S-chase} that  $R$ is an  $S$-coherent ring. 
	

	
\end{proof}

The following result was shown in \cite[Remark 3.4(1)]{bh18}. We exhibit  for completeness.
\begin{lemma}\cite[Remark 3.4(1)]{bh18}
	Suppose $R$ is an $S$-Noetherian ring. Then $R$ is $S$-coherent.
\end{lemma}
\begin{proof} Let $I$ be a finitely generated ideal of an $S$-Noetherian ring $R$. Then there is an exact sequence $0\rightarrow K\rightarrow F\rightarrow I\rightarrow 0$ with $F$ finitely generated free. Then $F$ is an $S$-Noetherian $R$-module by \cite{ad02}. Hence $K$ is $S$-finite, and so $I$ is $S$-finitely presented. Consequently, $R$ is an $S$-coherent ring.
\end{proof}
The next example shows that $S$-Noetherian rings, and hence $S$-coherent rings,  need not be $c$-$S$-coherent rings in general.
\begin{example}\cite[Remark 3.4(2)]{bh18}
Let $R = \mathbb{Z}(+)(\mathbb{Z}/2\mathbb{Z})^{(\mathbb{N})}$ and the multiplicative
set $S =\{(2, 0)^n\mid n\geq 0\}$.  Then for any ideal $I$ of $R$, we have $(2, 0)I$ is finitely generated. Indeed, $(2, 0)I=2J(+)0$, where $J=\{a\in\mathbb{Z}\mid \mbox{there is}\ b\in (\mathbb{Z}/2\mathbb{Z})^{(\mathbb{N})}\ \mbox{such that}\ (a,b)\in I\}$. Suppose $J=a\mathbb{Z}$ for some $a\in \mathbb{Z}$. Then $(2, 0)I=\langle (2a,0)\rangle$. Thus $R$
is an $S$-Noetherian ring.

 Next, we will show $R$ is not a $c$-$S$-coherent ring. Indeed, we consider the ideal generated by $(2, 0)$. Since $((0, 0):_R(2, 0))=0(+)(\mathbb{Z}/2\mathbb{Z})^{(\mathbb{N})}$, we have $(2, 0)$ is not finitely presented. Note that $\langle(2, 0)\rangle=\{(2m,0)\mid m\in \mathbb{Z}\}$. So  ideals of $R$ contained in $\langle(2, 0)\rangle$ are exactly of the form  $\langle(2m, 0)\rangle$ or $\langle(0, 0)\rangle$ with $m\geq 1$. Also $((0, 0):_R(2m, 0))=0(+)(\mathbb{Z}/2\mathbb{Z})^{(\mathbb{N})}$, and thus ideals of the form $\langle(2m, 0)\rangle$ are not finitely presented. Since $(2, 0)^n(\langle(2, 0)\rangle/\langle(0, 0)\rangle)\cong (2, 0)^n\langle(2, 0)\rangle=\langle(2^{n+1}, 0)\rangle\not=0$ for all $n\geq 1$, hence  $\langle(2, 0)\rangle$ is not $c$-$S$-finitely presented. Consequently, $R$ is not a $c$-$S$-coherent ring.
\end{example}

\end{document}